\documentclass{amsart}
\usepackage[utf8]{inputenc}
\usepackage{amsfonts}
\usepackage{hyperref}

\usepackage{amsmath}
\usepackage{amssymb}
\usepackage{xcolor}
\usepackage{amsthm}
\usepackage{pdflscape}
\usepackage{mathrsfs}
\usepackage{graphicx}
\newtheorem{thm}{Theorem}[section]
\newtheorem{cor}[thm]{Corollary}
\newtheorem{prop}[thm]{Proposition}
\newtheorem{lem}[thm]{Lemma}

\theoremstyle{definition}
\newtheorem{defn}[thm]{Definition}

\newtheorem{rem}[thm]{Remark}

\newtheorem{innercustomthm}{Theorem}
\newenvironment{customthm}[1]
  {\renewcommand\theinnercustomthm{#1}\innercustomthm}
  {\endinnercustomthm}

\newcommand{\C}{\mathbb{C}}
\newcommand{\R}{\mathbb{R}}

\newcommand{\N}{\mathbb{N}}
\newcommand{\norm}[1]{\left\lVert#1\right\rVert}

\makeatletter
\let\c@equation\c@thm
\makeatother
\numberwithin{equation}{section}

\title[Self-adjoint operators in $\mathcal{Z}$-stable C$^*$-algebras]{Self-adjoint operators in $\mathcal{Z}$-stable C$^*$-algebras with Prescribed Spectral Data}

\author{Andrew S. Toms and Hao Wan}
\date{\today}
\begin{document}
\begin{abstract}
We consider the variety of spectral measures that are induced by quasitraces on the spectrum of a self-adjoint operator in a simple separable unital and $\mathcal{Z}$-stable C$^*$-algebra.  This amounts to a continuous map from the simplex of quasitraces of the C$^*$-algebra into regular Borel probability measures on the spectrum of the operator under consideration.  In the case of a connected spectrum this data determines the unitary equivalence class of the operator, and may be reduced to to the case of an operator with spectrum equal to the closed unit interval.  We prove that any continuous map from the simplex of quasitraces with the topology of pointwise convergence into regular faithful Borel probability measures on $[0,1]$ with the Levy-Prokhorov metric is realized by some self-adjoint operator in the C$^*$-algebra.    

\end{abstract}

\maketitle
% \tableofcontents

\section{Introduction} \label{section: intro}
It is a classical result of H. Weyl that the distance between unitary equivalence classes of a self-adjoint operators on a finite-dimensional vector space is determined by spectral data:  the normalized trace on $n \times n$ matrices induces a probability measure on the spectrum of the operator, and the distance between two unitary equivalence classes is computed as the L\'evy-Prokhorov distance between the corresponding measures.  This classification is completed by the easy fact that any average of $n$ Dirac probability measures supported on $\mathbb{R}$ can occur as the spectral measure of such an operator.  

Weyl's result admits a remarkable degree of generalization to infinite-dimensional C$^*$-algebras.  For instance, Azoff and Davis proved that for self-adjoint operators in $B(H)$ for a Hilbert space $H$, approximate unitary equivalence classes are completely determined by the so-called crude multiplicity function which captures spectral data \cite{AD84}.  Sherman later generalized this result to normal operators in an arbitrary von Neumann algebra $\mathcal{M}$ \cite{She07}. 

More recently such results have pushed into the realm of separable C$^*$-algebras which are $\mathcal{Z}$-stable, an essential condition in Elliott's classification program for separable amenable C$^*$-algebras.  
The Riesz-Markov-Kakutani Theorem entails that every normalized quasitrace on a unital C$^*$-algebra $A$ induces a regular Borel probability measure on the spectrum of each normal operator in $a \in A$.  
Taken collectively these measures can be viewed as a single continuous affine map $f_a$ from the simplex of normalized quasitraces on $A$ into probability measures on the spectrum $\sigma(a)$ equipped with the weak topology, a topology which is metrizable via the L\'evy-Prokhorov metric.  Jacelon, Strung and the first author proved in \cite{JST15} that the distance $d_U(a,b)$ between the unitary orbits of self-adjoint operators $a$ and $b$ with unit interval spectrum in a unital simple separable C$^*$-algebra $A$ which is exact and $\mathcal{Z}$-stable is
\[
\sup_{\tau \in \mathrm{T(A)}} d_{LP}(f_a(\tau),f_b(\tau)),
\]
where $d_{LP}$ denotes the L\'evy-Prokhorov distance and $\mathrm{T}(A)$ denotes the tracial state space of $A$.  Briefly, the answer is again spectral.  Missing, however, is the range of possibilities for $f_a$.  Here we close this gap using tools from the theory of the Cuntz semigroup:  

\begin{thm}\label{main}
    Let $A$ be unital simple separable $\mathcal{Z}$-stable C$^*$-algebra with nonempty quasitrace simplex $\mathrm{QT}(A)$.  Let $h$ be a continuous affine map from $\mathrm{QT}(A)$ into the set $\mathcal{P}([0,1])^+$ of faithful regular Borel probability measures on $[0,1]$.  It follows that there is a self-adjoint element $a \in A$ with spectrum equal to $[0,1]$ such that $f_a = h$.  
\end{thm}
\noindent
We note that nuclearity and even exactness are not required here, whatever the ultimate answer to the question of whether quasitraces are traces might be.
Theorem \ref{main} is proved by thoroughly modern means:  we use the map $h$ to generate a Cuntz semigroup morphism $\gamma_h: \mathrm{Cu}(C[0,1]) \to \mathrm{Cu}(A)$ then appeal to a remarkably general lifting result of Robert to obtain a $*$-homomorphism $\Gamma_h:C[0,1] \to A$ with $\mathrm{Cu}(\Gamma_h) = \gamma_h$.  The desired operator $a$ is then $\Gamma_h(\mathrm{id})$.  Along the way we give characterizations of compact containment in $\mathrm{Cu}(C[0,1])$ which we expect will have further applications.

The paper is organized as follows:  Section \ref{section: prelim} recalls essential results from measure theory and the theory of the Cuntz semigroup;  Section \ref{section: cc} gives new characterizations of the compact containment relation
in the Cuntz semigroup of certain C$^*$-algebras via rank functions; Section \ref{section: proof of main} constructs the Cuntz semigroup morphism which is then lifted to a $*$-homomorphism to prove Theorem \ref{main}.

\section{Preliminaries} \label{section: prelim}

\subsection{The L\'evy-Prokhorov metric and weak convergence}
\begin{defn}[{\cite{Bil99}}]
Let $(M, d)$ be a metric space with its Borel $\sigma$-algebra $\mathcal{B}(M)$. Let $\mathcal{P}(M)$ denote the collection of all regular probability measures on the measurable space $(M, \mathcal{B}(M))$.  We say that $\mu \in \mathcal{P}(M)$ is {\it faithful} if $\mu(U) >0$ for every nonempty open subset $U$ of $M$, and denote the set of all such measures by $\mathcal{P}(M)^+$.
For a subset $A \subseteq M$, define the $\varepsilon$-neighborhood of $A$ by
\[A^{\varepsilon}:=\{p \in M \mid \exists q \in A, d(p, q)<\varepsilon\}.\]
The \textit{L\'evy-Prokhorov metric} $d_{LP}: \mathcal{P}(M)^2 \rightarrow[0,+\infty)$ is defined by \[d_{LP}(\mu, \nu):=\inf \left\{\varepsilon>0 \mid \mu(A) \leq \nu\left(A^{\varepsilon}\right)+\varepsilon \text { and } \nu(A) \leq \mu\left(A^{\varepsilon}\right)+\varepsilon, \forall A \in \mathcal{B}(M)\right\}.\]
\end{defn}

There are several equivalent definitions of weak convergence of a sequence of measures, and the equivalence of these conditions is known as the Portmanteau Theorem, which we recall here.

\begin{thm}[{\cite[Theorem 2.1]{Bil99}}]
    
Let $(M, d)$ be a metric space with its Borel $\sigma$-algebra $\mathcal{B}(M)$. Let $(P_n)_{n\in \N}$ be a sequence of probability measures on $(M, \mathcal{B}(M))$.  Let $\mathrm{E}_n$ denote the expectation or $L^1$ norm with respect to $P_n$, while E denotes expectation or the $L^1$ norm with respect to $P \in \mathcal{P}(M)$.  The following are equivalent:
\begin{enumerate}
    \item $\mathrm{E}_n[f] \rightarrow \mathrm{E}[f]$ for all bounded, continuous functions $f$;
    \item $\mathrm{E}_n[f] \rightarrow \mathrm{E}[f]$ for all bounded and Lipschitz functions $f$;
    \item $\lim \sup \mathrm{E}_n[f] \leq \mathrm{E}[f]$ for every upper semi-continuous function $f$ bounded from above;
    \item $\liminf \mathrm{E}_n[f] \geq \mathrm{E}[f]$ for every lower semi-continuous function $f$ bounded from below;
    \item $\lim \sup P_n(C) \leq P(C)$ for all closed sets $C$ of space $M$;
    \item $\liminf P_n(U) \geq P(U)$ for all open sets $U$ of space $M$;
    \item $\lim P_n(A)=P(A)$ for all continuity sets $A$ of measure $P$.
\end{enumerate}
If $(P_n)$ and $P$ satisfy any of (1) through (7) above then we say that $(P_n)$ converges weakly to $P$.
\end{thm}

\begin{thm}[{\cite[Theorem 6.8]{Bil99}}] \label{weak conv equiv to conv in L-P metric}
    If $(M, d)$ is separable and complete, then convergence of measures in the L\'evy-Prokhorov metric is equivalent to weak convergence of measures.
\end{thm}

\subsection{The Riesz-Markov-Kakutani Representation Theorem}
    Recall the following representation theorem due to Riesz, Markov, and Kakutani.
\begin{thm}[{\cite[Theorem 2.14]{Rud87}}]

    Let $X$ be a locally compact Hausdorff space. For any continuous linear functional $\psi$ on $C_0(X)$, there is a unique complex-valued regular Borel measure $\mu$ on $X$ such that
$$
\psi(f)=\int_X f(x) d \mu(x), \quad \forall f \in C_0(X)
$$
A complex-valued Borel measure $\mu$ is called regular if the positive measure $|\mu|$ is regular. The norm of $\psi$ as a linear functional is the total variation of $\mu$, that is
$$
\|\psi\|=|\mu|(X)
$$
Finally, $\psi$ is positive if and only if the measure $\mu$ is positive.
\end{thm}

We are interested in the special case where $X$ is a compact metric space and $\psi$ is a tracial state on $C(X)$.

\begin{cor}
    Let $X$ be a compact metric space. For any state $\tau$ on $C(X)$ (which is automatically tracial), there is a unique regular Borel probability measure $\mu$ on $X$ such that
$$
\tau(f)=\int_X f(x) d \mu(x), \quad \forall f \in C(X)
$$
Conversely, if $\mu$ is a regular Borel probability measure on $X$, then 
$$
\tau(f)=\int_X f(x) d \mu(x), \quad \forall f \in C(X)
$$ defines a tracial state $\tau$ on $C(X)$.
% \textcolor{red}{Moreover, the space $\mathcal{P}(X)$ of probability measures on $X$ is homeomorphic to the tracial state space $\mathrm{T}(C(X))$. }
This bijective correspondence in fact defines a homeomorphism between $(\mathcal{P}(X),d_{LP})$ and $(\mathrm{T}(C(X)),w^{*})$.
\end{cor}

\subsection{The Cuntz Semigroup}

We refer the reader to \cite{GP24} and \cite{APT11} for the basic theory of the Cuntz semigroup. For the convenience of the reader we recall here the aspects of this theory that are essential to the sequel.

\begin{defn}[{\cite[Definition 2.1]{GP24}}]
    Let $A$ be a C$^*$-algebra and let $a, b \in A_{+}$. We say that $a$ is \textit{Cuntz subequivalent to} $b$ (in $A$), denoted $a \precsim b$ (or $a \precsim_{A} b$ if there is a need to specify the ambient $\mathrm{C}^*$-algebra), if there exists a sequence $\left(r_n\right)_{n \in \mathbb{N}}$ in $A$ such that $\lim _{n \rightarrow \infty}\left\|r_n b r_n^*-a\right\|=0$. Equivalently, for every $\varepsilon>0$ there exists $r \in A$ such that $\left\|r b r^*-a\right\|<\varepsilon$.
We say that $a$ and $b$ are \textit{Cuntz equivalent}, written $a \sim b$, if $a \precsim b$ and $b \precsim a$.
\end{defn}

For a function $f: X \longrightarrow \R$, we denote its open support by
\[\operatorname{supp}_{\mathrm{o}}(f):=\{x \in X: f(x) \neq 0\}.\]

\begin{prop}[{\cite[Proposition 2.2]{GP24}}] \label{<= for elements in C(X)}
    Let $X$ be a locally compact, Hausdorff space, and let $f, g \in$ $C_0(X)$ be positive functions. Then $f \precsim g$ if and only if
\[\operatorname{supp}_{\mathrm{o}}(f) \subseteq \operatorname{supp}_{\mathrm{o}}(g).\]
Equivalently, $g(x)=0$ implies $f(x)=0$, for all $x \in X$.
\end{prop}

\begin{lem}[{\cite[Lemma 2.5]{GP24}}] \label{norm close subequivalence}
    Let $A$ be a $\mathrm{C}^*$-algebra, let $\varepsilon>0$, and let $a, b \in A_{+}$ with $\|a-b\|<$ $\varepsilon$. Then there exists a contraction $r \in A$ with $r b r^*=(a-\varepsilon)_{+}$. In particular, $(a-\varepsilon)_{+} \precsim b$.
\end{lem}

\begin{defn}[{\cite[Definition 2.9]{GP24}}]
    A unital $\mathrm{C}^*$-algebra $A$ is said to be \textit{finite} if for all $s \in A$ with $s^* s=1_A$ we automatically have $s s^*=1_A$. In other words, every isometry in $A$ is a unitary. We say that $A$ is \textit{stably finite} if $M_n(A)$ is finite for all $n \in \mathbb{N}$.
\end{defn}

\begin{defn}[{\cite[Definition 2.13]{GP24}}]
    A unital $\mathrm{C}^*$-algebra $A$ is said to have \textit{stable rank one} if the set $\mathrm{GL}(A)$ of invertible elements in $A$ is dense in $A$. A nonunital $\mathrm{C}^*$-algebra is said to have stable rank one if its minimal unitization does.
\end{defn}

\begin{defn}[{\cite[Definition 3.1]{GP24}}]
    Let $A$ be a $\mathrm{C}^*$-algebra. The \textit{Cuntz semigroup} of $A$ is defined as
\[\mathrm{Cu}(A)=(A \otimes \mathcal{K})_{+} / \sim,\]
where $\sim$ stands for the Cuntz equivalence relation. For a positive element $a \in$ $(A \otimes \mathcal{K})_{+}$, we denote by $[a]$ its Cuntz equivalence class, hence \begin{align*}
    \operatorname{Cu}(A)=\{[a]: a \in \left.(A \otimes \mathcal{K})_{+}\right\}.
\end{align*} There is a natural partial order defined on $\mathrm{Cu}(A)$, namely $[a] \leq[b]$ if $a \precsim b$. The element $0:=[0]$ is the minimum element in $\mathrm{Cu}(A)$. The addition on $\mathrm{Cu}(A)$ is defined by setting
\[[a]+[b]=\left[\left(\begin{array}{ll}
a & 0 \\
0 & b
\end{array}\right)\right].\]
We also denote $\left(\begin{array}{ll}a & 0 \\ 0 & b\end{array}\right)$ by $a \oplus b$.
\end{defn}

\begin{thm}[{\cite[Theorem 3.8]{GP24}}]
    Let $A$ be a C$^*$-algebra. Then every increasing sequence in $\mathrm{Cu}(A)$ has a supremum.
\end{thm}

\begin{defn}[{\cite[Definition 4.1]{GP24}}]
    Let $(S, \leq)$ be an ordered set. We define an additional relation $\ll$ on $S$, called \textit{(sequential) compact containment}, as follows: for $x, y \in S$, we set $x\ll y$ if whenever $\left(z_n\right)_{n \in \mathbb{N}}$ is an increasing sequence in $S$ with supremum $z$ satisfying $y\leq z$, then there exists $n_0 \in \mathbb{N}$ such that $x \leq z_{n_0}$.
    We say that $s \in S$ is compact if $s \ll s$, and noncompact otherwise.
\end{defn}

\begin{rem} \label{rem: either compact}
    Observe, from the definition of compact containment, that
    \begin{enumerate}
        \item If $x\leq y$ and $x \ll x$, then $x\ll y$.
        \item If $x\leq y$ and $y\ll y$, then $x\ll y$.
        \item If $x\ll y$, then $x\leq y$.
    \end{enumerate}
    Therefore, if either $x$ or $y$ is compact, then $x\ll y$ is equivalent to $x\leq y$.
\end{rem}
Given $\epsilon>0$, let $l_{\epsilon}: [0, \infty) \longrightarrow [0,\infty)$ be given by $l_{\epsilon}(s)=\max\{0, s-\epsilon\}$. Given $a\in A_{+}$ we write $(a-\epsilon)_{+}$ for $l_{\epsilon}(a)$. The element $(a-\epsilon)_{+}$ is usually refereed to as the \textit{$\epsilon$ cut-down of $a$}.
\begin{prop}[{\cite[Proposition 4.3]{GP24}}] \label{epsilon cc}
   Let $A$ be a $\mathrm{C}^*$-algebra and let $a, b \in(A \otimes \mathcal{K})_{+}$. Then $[a] \ll[b]$ if and only if there exists $\varepsilon>0$ such that $[a] \leq\left[(b-\varepsilon)_{+}\right]$. In particular, $[a]$ is compact if and only if there exists $\varepsilon>0$ with $a \sim(a-\varepsilon)_{+}$.
\end{prop}

\begin{prop}[{\cite[Proposition 4.4]{GP24}}]
    Let $X$ be a compact Hausdorff space and let $a, b \in C(X)_{+}$. Then $[a] \ll[b]$ if and only if $$
\overline{\operatorname{supp}_{\mathrm{o}}(a)} \subseteq \operatorname{supp}_{\mathrm{o}}(b) .
$$
In particular, for $a$ as above, $[a]$ is compact if and only if $\operatorname{supp}_{\mathrm{o}}(a)$ is compact in the lattice of open sets of $X$.
\end{prop}

We note that the foregoing Proposition informs the terminology of compact containment.

\begin{defn}[{\cite[Definition 4.5]{GP24}}]
     Let $(S, \leq)$ be a positively ordered monoid. We say that $S$ is an \textit{(abstract) Cuntz semigroup}, or just a \textit{$\mathbf{Cu}$-semigroup}, if it satisfies the following axioms:
     \begin{enumerate}
         \item Every increasing sequence has a supremum.
         \item For every $s \in S$, there is a sequence $\left(s_n\right)_{n \in \mathbb{N}}$ in $S$ with $s_n \ll s_{n+1}$ and $s=\sup _{n \in \mathbb{N}} s_n$.
         \item If $s \ll t$ and $s^{\prime} \ll t^{\prime}$, then $s+s^{\prime} \ll t+t^{\prime}$.
         \item If $\left(s_n\right)_{n \in \mathbb{N}}$ and $\left(t_n\right)_{n \in \mathbb{N}}$ are increasing sequences, then \[\sup _{n \in \mathbb{N}}\left(s_n+t_n\right)=\sup _{n \in \mathbb{N}} s_n+\sup _{n \in \mathbb{N}} t_n .\]
     \end{enumerate}
     
Given $\mathbf{Cu}$-semigroups $S$ and $T$, a \textit{$\mathbf{Cu}$-morphism} between them is a map \mbox{$f: S \rightarrow T$} preserving addition, neutral element, order $\leq$, suprema of increasing sequences, and also the compact containment relation $\ll$. 

We denote by $\mathbf{Cu}$ the category whose objects are $\mathbf{Cu}$-semigroups and whose morphisms are $\mathbf{Cu}$-morphisms. The set of $\mathbf{Cu}$-morphisms between two semigroups $S$ and $T$ will be denoted by $\mathrm{Cu}(S, T)$.
\end{defn}

The following result, due to Coward, Elliott and Ivanescu \cite{CEI08} reconciles the abstract and concrete Cuntz semigroups.

\begin{thm}[{\cite[Theorem 4.6]{GP24}}]
    Let $A$ be a $\mathrm{C}^*$-algebra. Then $\mathrm{Cu}(A)$ is a \mbox{$\mathbf{Cu}$-semigroup}. Moreover, if $\varphi: A \rightarrow B$ is a $*$-homomorphism between $\mathrm{C}^*$-algebras, then $\varphi$ naturally induces a $\mathbf{Cu}$-morphism 
    $\mathrm{Cu}(\varphi): \mathrm{Cu}(A) \rightarrow \mathrm{Cu}(B)$. In other words, Cu is a functor from the category $\mathbf{C}^*$ of $\mathrm{C}^*$-algebras to $\mathbf{Cu}$.
\end{thm}

% If $\Phi: C \longrightarrow D$ is a $*$-homomorphism then $\mathrm{Cu}(\Phi): \mathrm{Cu}(C) \longrightarrow \mathrm{Cu}(D)$ is defined by $\mathrm{Cu}(\Phi)([c])=[\Phi(c)]$. In particular, if $[c]$ is noncompact, then $[\Phi(c)]$ can be identified by the map $(\tau 
% \longrightarrow d_{\tau}(\Phi(c)))$.

\begin{defn}[{\cite[Definition 6.2]{GP24}}]
     Let $A$ be a C$^*$-algebra. A 1-quasitrace is a function $\tau: A \rightarrow \mathbb{C}$ such that
     \begin{enumerate}
         \item $\tau(a+i b)=\tau(a)+i \tau(b)$ for all $a, b \in A_{sa}$;
         \item $\tau$ is linear on commutative subalgebras of $A$;
         \item $\tau\left(a^* a\right)=\tau\left(a a^*\right) \geq 0$ for all $A$.
     \end{enumerate}
     If $A$ is unital, we require $\tau(1)=1$ as well.
     We say that $\tau$ is a quasitrace if it extends to a 1-quasitrace on $M_n(A)$ for all $n \in \mathbb{N}$. We use $\mathrm{QT}(A)$ to denote the space of all quasitraces on $A$.
\end{defn}

\begin{defn}[{\cite[Definition 6.4]{GP24}}]
    Let $A$ be a unital $\mathrm{C}^*$-algebra. Given $\tau \in \mathrm{QT}(A)$, we define its associated dimension function $d_\tau: \mathrm{Cu}(A) \rightarrow[0, \infty]$ by
    \[d_\tau([a])=\lim _{n \rightarrow \infty} \tau\left(a^{\frac{1}{n}}\right)\]
    for all $a \in M_{\infty}(A)_{+}$, and extended to $(A \otimes \mathcal{K})_{+}$by taking suprema. 

\end{defn}
Implicit in the definition is the fact that $d_\tau([a])$ only depends on $[a]$. As such we will use $d_\tau([a])$ and $d_\tau(a)$ interchangeably.

\begin{prop}[{\cite[Proposition 6.6]{GP24}}]
    Let $A$ be a unital C$^*$-algebra. Given $\tau \in \mathrm{QT}(A)$, the map $d_\tau$ is well-defined on $\mathrm{Cu}(A)$. Moreover, we have:
    \begin{enumerate}
        \item $d_\tau(0)=0$;
        \item $d_\tau([1])=1$;
        \item $d_\tau(s+t)=d_\tau(s)+d_\tau(t)$ for all $s, t \in \mathrm{Cu}(A)$;
        \item $d_\tau(s) \leq d_\tau(t)$ whenever $s \leq t$ in $\mathrm{Cu}(A)$;
        \item $d_\tau$ preserves suprema of increasing sequences.
    \end{enumerate}
\end{prop}

\begin{prop}[{\cite[Proposition 6.7]{GP24}}] \label{rank, measure of suppport}
    Let $X$ be a compact Hausdorff space, let $\mu$ be a regular Borel probability measure, and let $\tau: C(X) \rightarrow \mathbb{C}$ be the trace it induces. Given $a \in$ $C(X)_{+}$, we have
\[d_{\tau}([a])=\mu\left(\operatorname{supp}_{\mathrm{o}}(a)\right).\]
\end{prop}
% \begin{proof}
%     Without loss of generality, we may assume that $\|a\| \leq 1$. Note that $\left(a^{\frac{1}{n}}\right)_{n \in \mathbb{N}}$ is an increasing sequence which converges pointwise to the indicator function of $\operatorname{supp}_{\mathrm{o}}(a)$. Applying the dominated convergence theorem at the third step, we get
% $$
% d_\tau([a])=\lim _{n \rightarrow \infty} \tau\left(a^{\frac{1}{n}}\right)=\lim _{n \rightarrow \infty} \int_X a^{\frac{1}{n}} d \mu=\int_X \lim _{n \rightarrow \infty} a^{\frac{1}{n}} d \mu=\mu\left(\operatorname{supp}_o(a)\right),
% $$
% as desired.
% \end{proof}

% \begin{prop}[{\cite[Proposition 2.23]{APT11}}] \label{0 isolated or not included in spectrum}
%     Let A be a unital $\mathrm{C}^*$-algebra with stable rank one. Then, for $a \in M_{\infty}(A)$, the following are equivalent:
%     \begin{enumerate}
%         \item $[a]=[p]$, for a projection $p$,
%         \item  0 is an isolated point of $\sigma(a)$, or $0 \notin \sigma(a)$.
%     \end{enumerate}
% \end{prop}

\begin{defn}[{\cite[Definition 8.5]{GP24}}]
    Let $A$ be a simple unital $\mathrm{C}^*$-algebra. We say that $A$ has \textit{strict comparison} (of positive elements by quasitraces) if whenever $a, b \in(A \otimes \mathcal{K})_{+}$ are nonzero and satisfy $d_\tau(a)<d_\tau(b)$ for all $\tau \in \operatorname{QT}(A)$, then $a \precsim b$.
\end{defn}

\begin{prop}[{\cite[Proposition 2.23]{APT11}}] \label{0 accum pt for noncompact}
    Let A be a unital C$^*$-algebra with stable rank one. It follows that $[a]$ is compact in $\mathrm{Cu}(A)$ if and only if $0$ is either an isolated point of $\sigma(a)$ or $0 \notin \sigma(a)$.
\end{prop}

\begin{cor}[{\cite[Corollary 2.24]{APT11}}] \label{substraction in Cu(A)}
    Let A be a unital C$^*$-algebra which is either simple or of stable rank one. Then the Murray-von Neumann semigroup $V(A)$ of $A$ can be identified as a subsemigroup of $\mathrm{Cu}(A)$:
    \[\mathrm{V}(A)=\{x \in \mathrm{Cu}(A) \mid \text{if } x\leq y \text{ for } y \in \mathrm{Cu}(A), \text{ then } x+z=y \text{ for some } z\in  \mathrm{Cu}(A)\}.\]
\end{cor}

\begin{prop}[{\cite[Proposition 5.9]{APT11}}] \label{qt: noncompact compact}
    Let A be a simple C$^*$-algebra with strict comparison. Let $[a]$ and $[b]$ be elements in  $\mathrm{Cu}(A)$ such that $[a]$ is noncompact and $[b]$ is compact. If $d_{\tau}(a)\leq d_{\tau}(b)$ for every $\tau \in \mathrm{QT}(A)$, then $a \precsim b$. 
\end{prop}

\begin{prop}[{\cite[Proposition 5.10]{APT11}}] \label{qt: compact noncompact}
    Let A be a simple C$^*$-algebra with strict comparison. Let $[a]$ and $[b]$ be elements in  $\mathrm{Cu}(A)$ such that $[a]$ is compact and $[b]$ is noncompact. Then, $a \precsim b$ if and only if $ d_{\tau}(a)< d_{\tau}(b)$ for every $\tau \in \mathrm{QT}(A)$.
\end{prop}

We summarize some properties of the C$^*$-algebras under consideration, using results in \cite{Cun78} and \cite{Ror04}.
\begin{thm}
    Let $A$ be a simple, separable, unital and $\mathcal{Z}$-stable C$^*$-algebra. 
    \begin{enumerate}
        \item $A$ is stably finite if and only if $QT(A)\neq \emptyset$.
        \item If $\mathrm{QT}(A) \neq \emptyset$, then $A$ has stable rank one and has strict comparison.
    \end{enumerate}
\end{thm}

The following theorem is a special case of Theorem 1.0.1 in \cite{Rob12}, which allows us to lift a $\mathbf{Cu}$-morphism to a $*$-homomorphism.
\begin{thm} \label{Cu-morphism lifts to $*$-homomorphism} Let $A$ be a C$^*$-algebra of stable rank one. 
It follows that for every $\mathbf{Cu}$-morphism $\gamma: \mathrm{Cu}(C[0,1]) \longrightarrow \mathrm{Cu}(A)$ satisfying $\gamma([1_{C[0,1]}]) \leq [1_A]$, there exists a $*$-homomorphism $\Gamma: C[0,1] \longrightarrow A$ such that $\mathrm{Cu}(\Gamma)=\gamma$. Moreover, $\Gamma$ is unique up to approximate unitary equivalence.
\end{thm}

Finally, we need two structure theorems for the Cuntz semigroups of the C$^*$-algebras under consideration. The first of these is a special case of Theorem 1.1 in \cite{Rob13}. 
We denote by $\overline{\N}$ the set of natural numbers with 0 and $\infty$ adjoined, and we denote by $\mathrm{Lsc}([0,1], \overline{\N})$ the lower semicontinuous functions from $[0,1]$ to $\overline{\N}$.
\begin{thm} \label{Cu(C[0,1]) iso}
    The map $\theta: \mathrm{Cu}(C[0,1]) \longrightarrow \mathrm{Lsc}([0,1], \overline{\N})$ defined by \[\theta([a])(t)=\mathrm{rank}(a(t))\] for all $a \in (C[0,1] \otimes K)_{+}$ is a $\mathbf{Cu}$-isomorphism.
\end{thm}

The second isomorphism theorem  was originally obtained in \cite{BPT08} and \cite{BT07} for a version of the Cuntz semigroup based on $\mathrm{M}_\infty(A)$, and was extended in \cite{CEI08} to the version presented below. 

For each element $x=[a] \in \mathrm{Cu}(A)$, we denote by $\widehat{x}: \mathrm{QT}(A) \rightarrow \mathbb{R}_{++}$ the lower semicontinuous function defined by $\widehat{x}(\tau)=d_{\tau}(a)$ for $\tau \in \mathrm{QT}(A)$. We also denote by $\mathrm{LAff}(\mathrm{QT}(A))_{++}$ the semigroup of lower semicontinuous affine functions defined on $\mathrm{QT}(A)$ with values on $(0, \infty]$.

\begin{thm}[{\cite[Theorem 9.7]{GP24}}] \label{Cu(A) iso}
    Let $A$ be a simple, separable, unital, stably finite $\mathcal{Z}$-stable $\mathrm{C}^*$-algebra. Then
\[\mathrm{Cu}(A) \cong \underbrace{\mathrm{V}(A)}_{\text {compacts }} \sqcup  \operatorname{LAff}(\mathrm{QT}(A))_{++},\]
with addition and order defined as follows:

\begin{enumerate}
    \item The addition in $\mathrm{V}(A)$ is the usual addition and in $\operatorname{LAff}(\mathrm{QT}(A))_{++}$is given by pointwise addition of functions. If $x \in \mathrm{~V}(A)$ and $f \in \operatorname{LAff}(\mathrm{QT}(A))_{++}$, then $x+f=\widehat{x}+f \in \operatorname{LAff}(\mathrm{QT}(A))_{++}$.
    \item For $x \in \mathrm{~V}(A)$ and $f \in \operatorname{LAff}(\mathrm{QT}(A))_{++}$, we have
    \begin{enumerate}
        \item $x \leq f$ if $\widehat{x}(\tau)<f(\tau)$ for every $\tau \in \mathrm{QT}(A)$.
        \item $f \leq x$ if $f(\tau) \leq \widehat{x}(\tau)$ for every $\tau \in \mathrm{QT}(A)$.
    \end{enumerate}
\end{enumerate}
\end{thm}

\section{Characterization of compact containment} \label{section: cc}
Here we characterize compact containment in $\mathrm{Cu}(C[0,1])$ in terms of the rank functions appearing in Theorem \ref{Cu(C[0,1]) iso}. This will be used to prove that a map constructed in Section \ref{section: proof of main} preserves compact containment and is therefore a \mbox{$\mathbf{Cu}$-morphism}. We recall first the following general topology result, known as the Katětov–Tong Insertion Theorem \cite{Ton52}.
\begin{thm} \label{continuous btw usc and lsc}
    If $X$ is a normal topological space and $f,g: X \longrightarrow \R$ are functions such that $f$ is upper semicontinuous, $g$ is lower semicontinuous and $f\leq g$, then there exists a continuous function $h: X \longrightarrow \R$ such that $f\leq h\leq g$.
\end{thm}

\begin{lem} \label{cutdown bounded}
    Let $A$ be a C*-algebra. For any $[b]\in \mathrm{Cu}(A)$ and $\epsilon>0$, the function \[t \mapsto \mathrm{rank}((b(t)-\epsilon)_{+})\] is bounded.
    In particular, if $[a]$ and $[b]$ are elements in $\mathrm{Cu}(A)$ such that $[a] \ll [b]$, then the function $t \mapsto \mathrm{rank}((a(t))$ is bounded.
    % $t \mapsto \mathrm{rank}((b(t)-\frac{\epsilon}{2})_{+})$ is bounded, that is, there exists $N \in \N$ such that $\mathrm{rank}((b(t)-\frac{\epsilon}{2})_{+}) \leq N$ for all $t \in [0,1]$.
\end{lem}
\begin{proof}  
    We know $C[0,1] \otimes \mathcal{K}$ is the inductive limit of $C[0,1] \otimes M_n(\C)$ under the inclusion map, meaning that, given $b \in C[0,1] \otimes \mathcal{K}$ and $\epsilon>0$, there exist $N \in \N$ and $c\in C[0,1] \otimes M_N(\C)$ such that $\norm{b-c}<\epsilon$. 
    In fact, if $b$ is positive, we can choose $c$ to be positive, too. By Lemma \ref{norm close subequivalence}, if $\norm{b-c}<\epsilon$ and $b,c\geq0$, then $(b-\epsilon)_{+} \precsim c$. 
    By Theorem \ref{Cu(C[0,1]) iso}, we have \[\mathrm{rank}\left(\left(b(t)-\epsilon\right)_{+}\right) \leq \mathrm{rank}(c(t))\leq N, \ \forall t \in [0,1]\] since each $c(t)$ is an $N \times N$ matrix.
    
    In particular, if $[a]\ll[b]$ then Proposition \ref{epsilon cc} says there exists $\epsilon>0$ such that $a \precsim (b-\epsilon)_+$, and Theorem \ref{Cu(C[0,1]) iso} gives $\mathrm{rank}(a(t)) \leq \mathrm{rank}\left(\left(b(t)-\epsilon\right)_{+}\right) \leq N$.
\end{proof}
\begin{lem} \label{strict < on open}
    Let $ x$ and $y$ be elements in $\mathrm{Cu}(C[0,1])$, and let $f_x$ and $f_y$ be the rank functions in $\mathrm{Lsc}([0,1], \overline{\N})$ associated with x and y, respectively, (meaning that if $x=[a]$ then $f_x(t)=\theta([a])(t)=\mathrm{rank}(a(t))$. 
    \begin{enumerate}
        \item If $x \ll y$, then there exists a continuous function \[g: [0,1] \longrightarrow [0,\infty)\] such that $f_x(t) \leq g(t) \leq f_y(t)$ for all $t \in [0,1]$.
        \item If, moreover, $y$ is noncompact, then there exists a nonempty open interval $U \subset [0,1]$ such that $f_x(t) \leq g(t) < f_y(t)$ for all $t \in U$.
    \end{enumerate}
\end{lem}
\begin{proof}
    \textbf{(1)}. Let $x=[a]$ and $y=[b]$. Proposition \ref{epsilon cc} says if $[a]\ll[b]$ then there exists $\epsilon>0$ such that $a \precsim (b-\epsilon)_+$, and it is also clear that $(b-\epsilon)_+ \precsim (b-\frac{\epsilon}{2})_+$. Notice that $(b-\epsilon)_{+}(t)=(b(t)-\epsilon)_{+}$, so by Theorem \ref{Cu(C[0,1]) iso} and Lemma \ref{cutdown bounded}, there exists $N \in \N$ such that   
    \[
    \mathrm{rank}(a(t)) \leq \mathrm{rank}((b(t)-\epsilon)_{+})\leq \mathrm{rank}\left(\left(b(t)-\frac{\epsilon}{2}\right)_{+}\right) \leq N, \ \forall t \in [0,1].
    \] 
    In fact, all we need is the following weaker version of Lemma \ref{cutdown bounded} that says
    \[
    \mathrm{rank}\left( \left(b(t)-\frac{\epsilon}{2} \right)_{+} \right) < \infty, \ \forall t \in [0,1],
    \] 
    which follows from the spectral theorem for compact normal operators.
    
    Define $r_\epsilon(t)$ to be the number of eigenvalues of $b(t)$ that are greater than or equal to $\epsilon$. Notice that $r_\epsilon(t)=N-\mathrm{rank}(s(b(t)))$ where $s: [0,1] \longrightarrow [0,\infty)$ is a continuous function whose open support is $[0,\epsilon)$.
    Since $\mathrm{rank}(s(b(t)))$ is lower semicontinuous, $r_\epsilon$ is upper semicontinuous. Counting eigenvalues with multiplicity we have the following statements: 
    \begin{align*}
        \mathrm{rank}((b(t)-\epsilon)_{+})&=\text{number of eigenvalues of $b(t)$ that are} >\epsilon \\
        r_\epsilon(t)&=\text{number of eigenvalues of $b(t)$ that are} \geq \epsilon, \\
        \text{and } \mathrm{rank}\left(\left(b(t)-\frac{\epsilon}{2}\right)_{+}\right)&=\text{number of eigenvalues of $b(t)$ that are} > \frac{\epsilon}{2}.
    \end{align*} 
    Therefore, for each $t \in [0,1]$, we have
    \begin{align*}
        \mathrm{rank}(a(t)) \leq \mathrm{rank}((b(t)-\epsilon)_{+}) \leq r_\epsilon(t) \leq \mathrm{rank}\left(\left(b(t)-\frac{\epsilon}{2}\right)_{+}\right) \leq \mathrm{rank}(b(t)).
    \end{align*}
    Since $\mathrm{rank}\left( \left(b(t)-\frac{\epsilon}{2} \right)_{+} \right) < \infty, \ \forall t \in [0,1]$, we can apply Theorem \ref{continuous btw usc and lsc} to get a continuous function $g: [0,1] \longrightarrow[0,\infty)$ such that \[r_\epsilon(t) \leq g(t) \leq \mathrm{rank}\left(\left(b(t)-\frac{\epsilon}{2}\right)_{+}\right), \ \forall t\in[0,1].\] Clearly, $\mathrm{rank}(a(t)) \leq g(t) \leq \mathrm{rank}(b(t)) \ \forall t\in[0,1]$, as required.

    \vspace{2mm}
    \noindent
    \textbf{(2)}. Now suppose $y$ is noncompact.
    
    \vspace{2mm}
    \noindent
    \textbf{Step 1}: Show there exists a nonempty open interval $V \subset [0,1]$ such that for all $t \in[0,1]$, we have $\mathrm{rank}(a(t)) < \mathrm{rank}(b(t))$.

    \vspace{2mm}
    \noindent
    \textbf{Case 1}: If $\mathrm{rank}(b(t))=\infty$ for all $t \in [0,1]$, then by Lemma \ref{cutdown bounded} there exists $N \in \N$ such that $\mathrm{rank}(a(t)) \leq N$ for all $t \in [0,1]$, so $V=[0,1]$ in this case.

    \vspace{2mm}
    \noindent
    \textbf{Case 2}: Suppose $\mathrm{rank}(b(t))<\infty$ for some $t \in [0,1]$.  
    We know from Proposition \ref{epsilon cc} and Theorem \ref{Cu(C[0,1]) iso} that $[a] \ll [b]$ implies there exists some $\epsilon>0$ such that \mbox{$\mathrm{rank}(a(t)) \leq \mathrm{rank}((b-\epsilon)_{+}(t))$} and hence $\mathrm{rank}(a(t)) < \mathrm{rank}(b(t))$ for all $t\in V$. 
    It suffices to show, for all $\epsilon >0$, there exists a nonempty open subset $V$ of $[0,1]$ such that $\mathrm{rank}((b-\epsilon)_{+}(t))<\mathrm{rank}(b(t))$ for all $t \in V$.
    
    Let $\epsilon>0$ and $n \in \N$ be the lowest possible rank of $b(t)$. 
    Define \[E_n=\{t \in [0,1] \mid \mathrm{rank}(b(t))=n\}.\] 
    Since $y$ is noncompact, $f_y$ is not constant and therefore $E_n$ is not the entire closed interval $[0,1]$ and so $[0,1]-E_n$ is nonempty. We also observe that \[[0,1]-E_n=\{t \in [0,1] \mid \mathrm{rank}(b(t))>n\}\] is open, so $E_n$ is closed. If we pick a boundary point $t_1 \in \partial E_n$, then $t_1 \in E_n$ too, and so $\mathrm{rank}(b(t_1))=n$. 
    Say the spectrum $\sigma(b(t_1) )$ of $b(t_1)$ is a subset of $\{0, \lambda_1, \lambda_2, \cdots, \lambda_n\}$, where $0 < \lambda_1 \leq \lambda_2 \leq \cdots \leq \lambda_n$. (We used subset instead of equal as 0 may or may not be in the spectrum of $b(t_1)$.)
    By continuity of $b(t)$, given $\alpha \in (0,\epsilon)$, there exists a $\delta>0$ such that $|t-t_1|<\delta$ implies $\norm{b(t)-b(t_1)}<\alpha$. 
    Then $\sigma(b(t))$ is at most ``$\alpha$-away" from $\sigma(b(t_1))$, meaning that \[\sigma(b(t)) \subset (\cup_{i=1}^{n}(\lambda_i -\alpha, \lambda_i+\alpha)) \cup [0,\alpha).\] 
    Recall that $(b-\epsilon)_{+}(t)=(b(t)-\epsilon)_{+}$ is defined to be $l_{\epsilon}(b(t))$ via continuous functional calculus, where $l_{\epsilon}: [0,\infty) \longrightarrow [0, \infty)$ is defined by \mbox{$l_{\epsilon}(s)=\max\{0, s-\epsilon\}$}.
    By the spectral mapping theorem, \[\sigma((b-\epsilon)_{+}(t))=\sigma(l_{\epsilon}(b(t)))=l_{\epsilon}(\sigma(b(t))) \subset \cup_{i=1}^{n}(\lambda_i-\epsilon-\alpha, \lambda_i-\epsilon+\alpha).\] 
    Since $\alpha \in (0, \epsilon)$ can be arbitrarily small, this shows 
    $\mathrm{rank}((b-\epsilon)_{+}(t))\leq n$. 
    Now we take $W=(t_1 -\delta, t_1+\delta)$ and $V=W \cap ([0,1]-E_n)$. Since $E_n$ is closed, $V$ is open. $V$ is also nonempty, because $W$ and $[0,1]-E_n$ are both nonempty and $t_1 \in \partial E_n$ means every neighborhood of $t_1$ should intersect both $E_n$ and $[0,1]-E_n$. Since $V \subset [0,1]-E_n$, this shows $\mathrm{rank}((b-\epsilon)_{+}(t)) < n$ for all $t \in V$, and hence $\mathrm{rank}((b-\epsilon)_{+}(t))<\mathrm{rank}(b(t))$ for all $t \in V$. Without loss of generality, we can assume $V$ to be an open interval $(c,d)$.

    \vspace{2mm}
    \noindent
    \textbf{Step 2}: Given the open interval $V$ constructed in Step 1 and the continuous function $g$ on $[0,1]$ constructed in part (1) of the lemma, we show there exists an open interval $U \subset [0,1]$ and a continuous function $\Tilde{g}: [0,1] \longrightarrow \R$ such that $f_x(t) \leq \Tilde{g}(t) < f_y(t)$ for all $t \in U$.

    \vspace{2mm}
    \noindent
    \textbf{Case 1}: Suppose $g$ is identical to $f_x$ on $V$. Let $\varphi$ be a bump function on $V$ bounded by $0$ and $\frac{1}{2}$, that is, $\varphi \in C^{\infty}_{c}([0,1])$, $\operatorname{supp}_{\mathrm{o}}(\varphi)=V=(c,d)$ and \mbox{$0\leq\varphi(t)\leq \frac{1}{2}$} for all $t \in [0,1]$, take $\Tilde{g}=g+\varphi$, then $f_x<\Tilde{g}<f_y$ on $V$. Notice that the strict inequalities come from the fact that $0\leq\varphi\leq \frac{1}{2}$ and $f_y-f_x\geq 1$ on $V$, just take $U$ to be $V$ in this case.

    \vspace{2mm}
    \noindent
    \textbf{Case 2}: Suppose $g$ is identical to $f_y$ on $V$. Using the same $\varphi$ as in Case 1 to define $\Tilde{g}=g-\varphi$, we have $f_x<\Tilde{g}<f_y$ on $V$ similarly, and again we take $U$ to be $V$ in this case.

    \vspace{2mm}
    \noindent
    \textbf{Case 3}: Suppose $g$ is not identical to $f_x$ and $g$ is not identical to $f_y$ on $V$, then there exists a $t_0 \in V$ making $g(t_0)<f_y(t_0)$. Since $f_y-g$ is lower semicontinuous, the set $\{t\in [0,1] \mid  f_y(t)-g(t)>0\}$ is open, which means this set contains some open interval $U$ around $t_0$ such that $f_y(t)>g(t)$ on $U$. We know from part (1) that $f_x(t)\leq g(t)\leq f_y(t)$ for all $t \in [0,1]$, combined with this, we get $f_x(t)\leq g(t)< f_y(t)$ on $U$. Take $\Tilde{g}$ to be $g$ in this case.
\end{proof}

A local characterization of compact containment in $\mathrm{Cu}(C[0,1])$ was given in Proposition 5.5 of \cite{APS11}, and we record it below.  We do not require the definition of a countably based object in $\mathbf{Cu}$ here and instead note simply that every Cuntz semigroup of a separable C$^*$-algebra is countably based (see \cite{GP24}, the remark after Definition 11.16).

\begin{prop} \label{Local Cu(C[0,1]) cc}
    Let $M$ be a countably based semigroup in $\mathbf{Cu}$ and $X$ be a compact Hausdorff space with finite covering dimension. Given two lower semicontinuous functions $f$ and $g$ from $X$ to $M$, we have $g \ll f$ if and only if for every $t \in X$ there are an open neighborhood $U_t$ of $t$, and $c_t \in M$ such that $g(s) \leq c_t \ll f(s)$ for all $s \in U_t$.
\end{prop}

We show now that in $\mathrm{Cu}(C[0,1])$, this local characterization of compact containment is equivalent to a global characterization of compact containment, and in fact, we can say more in the case where both elements are noncompact.

\begin{prop} \label{Cu(C[0,1]) cc}

    Let $x$ and $y$ be elements in $\mathrm{Cu}(C[0,1])$, and let $f_x$ and $f_y$ be the functions in $\mathrm{Lsc}([0,1], \overline{\N})$ associated with x and y, respectively. The following are equivalent
    \begin{enumerate}
        \item $x\ll y$.
        \item For all $t \in [0,1]$, there exist an open interval $U_t$ around $t$ and $c_t \in \overline{\N}$ such that $f_x(s) \leq c_t \leq f_y(s)$ for all $s\in U_t$.
        \item There exists a continuous function $g: [0,1] \longrightarrow [0,\infty)$ such that \[f_x(t) \leq g(t) \leq f_y(t)\] for all $t \in [0,1]$.
    \end{enumerate} If, moreover,  either $x$ or $y$ is compact in $\mathrm{Cu}(C[0,1])$, then 
    \[
    x\ll y \iff f_x \leq f_y.
    \]
    Finally, if $y$ is noncompact, then $x\ll y$ implies that there is a nonempty open interval $U \subset [0,1]$ such that $f_x(t) \leq g(t) < f_y(t)$ for all $t \in U$.
\begin{proof}
    In Proposition \ref{Local Cu(C[0,1]) cc}, take $X=[0,1]$, $M=\overline{\N}$, it is not hard to see in $M=\overline{\N}$, compact containment is equivalent to $\leq$. This proposition says (1) $\iff$ (2).

    % ($\implies$) Suppose that for all $t \in [0,1]$, there exist an open interval $U_t$ around $t$ and $c_t \in \overline{\N}$ such that $f_x(s) \leq c_t \leq f_y(s)$ for all $s\in U_t$, we want to show exists a continuous and bounded function $g: [0,1] \longrightarrow [0,\infty)$ such that $f_x(t) \leq g(t) \leq f_y(t)$ for all $t \in [0,1]$. \textcolor{red}{details}
    Lemma \ref{strict < on open} gives (1) $\implies$ (3), and that if $x$ and $y$ are both noncompact, then \mbox{$x\ll y$} implies there exists a nonempty open interval $U \subset [0,1]$ such that \mbox{$f_x(t) \leq g(t) < f_y(t)$} for all $t \in U$. 
    
    Let us show (3)$\implies$(2). Suppose there exists a continuous function \[g: [0,1] \longrightarrow [0,\infty)\] such that $f_x(t) \leq g(t) \leq f_y(t)$ for all $t \in [0,1]$. We want to show for any $t \in [0,1]$, there are an open interval $U_t$ around $t$ and a $c_t \in \overline{\N}$ such that $f_x(s) \leq c_t \leq f_y(s)$ for all $s\in U_t$. 
    Pick an arbitrary $t\in[0,1]$, we define $c_t=f_y(t)$.

    \vspace{2mm}
    \noindent
    \textbf{Case 1}: If $f_x(t)<f_y(t)=c_t$, by lower semicontinuity of $f_y$, the set \[\{s\in [0,1] \mid f_y(s)>c_t-\frac{1}{2}\}\] is an open set containing $t$, so it contains an open interval $V_t$ around $t$, this means $f_y(s)>c_t-\frac{1}{2}$ for all $s \in V_t$. 
    We know $f_x$ is integer-valued, so $f_y(s)\geq c_t$ for all $s \in V_t$.
    By continuity of $g$, given $\epsilon=\frac{1}{2}$, there exists an open interval $W_t$ around $t$ such that $g(s) \leq g(t) + \frac{1}{2}$ for all $s \in W_t$, and so
    \begin{align*}
        f_x(s) \leq g(s) \leq g(t) + \frac{1}{2} \leq f_y(t) + \frac{1}{2} = c_t + \frac{1}{2}.
    \end{align*} Again, as $f_x$ is integer-valued, this means $f_x(s) \leq c_t$ for all $s\in W_t$. Set $U_t=V_t \cap W_t$, we have $f_x(s) \leq c_t \leq f_y(s)$ for all $s\in U_t$.

    \vspace{2mm}
    \noindent
    \textbf{Case 2}: If $f_x(t)=g(t)=f_y(t)=c_t$, by continuity of $g$, given $\epsilon=\frac{1}{2}$, there exists $\delta>0$ such that if $s \in U_t=(t-\delta,t+\delta)$ then $|g(s)-g(t)|<\frac{1}{2}$, that is, $c_t-\frac{1}{2}<g(s)<c_t+\frac{1}{2}$. Then $f_x(s)\leq g(s) < c_t+\frac{1}{2}$ and $f_y(s)\geq g(s) > c_t-\frac{1}{2}$. Since $f_x$ and $f_y$ are integer valued, we have $f_x(s)\leq c_t \leq f_y(s)$ for all $s\in U_t$.

    Finally, it follows from Remark \ref{rem: either compact} that if either $x$ or $y$ is compact, then \mbox{$x \ll y \iff x\leq y$}, and Theorem \ref{Cu(C[0,1]) iso} that $x \leq y \iff f_x \leq f_y$.
\end{proof}

\begin{lem} \label{cc lemma}
    Let $X$ be a compact Hausdorff space. Let $r$ and $g$ be functions from $X$ to $\R$ such that  $g$ is lower semicontinuous, $r$ is continuous, and $g(x) > r(x)$ for all $x \in X$. Let $(g_n)_{n \in \N}$ be an increasing sequence of lower semicontinuous functions from $X$ to $\R$ such that $\sup_{n \in \N}g_n(x) \geq g(x)$ for all $x\in X$. Then there exists $n_0 \in \N$ such that $g_{n_0}(x) > r(x)$ for all $x\in X$.
\end{lem}
\begin{proof}
    Since $g-r$ is lower semicontinuous and strictly positive, it achieves a minimum $m>0$ on X. Then $g(x)-r(x) \geq m >\frac{m}{2}>0$ for all $x\in X$. Define $\Tilde{g_n}(x)=\max\{0, r(x)-g_n(x)+\frac{m}{2}\}$, then $(\Tilde{g_n})_{n \in \N}$ is upper semicontinuous, decreasing, non-negative, and converges pointwise to 0. 
    By Dini's theorem for semicontinuous functions, the convergence is uniform. Therefore, given $\epsilon=\frac{m}{2}$, there exists $n_0 \in \N$ such that $0 \leq \Tilde{g_n}(x) <\frac{m}{2}$ for all $x \in X$ and $n \geq n_0$. 
    In particular, $0 \leq r(x)-g_{n_0}(x)+\frac{m}{2} < \frac{m}{2}$, and hence $g_{n_0}(x)>r(x)$ for all $x\in X$.
\end{proof}

{\rm We now give a characterization of compact containment in terms of rank functions.}

\begin{prop}\label{Cu(A) cc}
Let $A$ be a simple, separable, unital, $\mathcal{Z}$-stable C$^*$-algebra with $\mathrm{QT}(A) \neq \emptyset$ and let $x, y$ be elements in $\mathrm{Cu}(A)$. 

\begin{enumerate}
    \item If x and y are both compact and $x\neq y$, then \[x \ll y \iff d_{\tau}(x)<d_{\tau}(y) \quad \forall \tau \in \mathrm{QT}(A).\]
    \item  If x is compact and y is noncompact, then \[x \ll y \iff d_{\tau}(x)<d_{\tau}(y) \quad \forall \tau \in \mathrm{QT}(A).\]
    \item If x is noncompact and y compact, then \[x \ll y \iff d_{\tau}(x)\leq d_{\tau}(y) \quad \forall \tau \in \mathrm{QT}(A).\]
    \item If x and y are both noncompact, then the following are equivalent:
    \begin{enumerate}
        \item $x\ll y$
        \item There exists a continuous and bounded function $r: \mathrm{QT}(A) \longrightarrow (0,\infty)$ such that $\hat{x}(\tau):=d_{\tau}(x) < r(\tau) < d_{\tau}(y)=:\hat{y}(\tau)$ for all $\tau \in \mathrm{QT}(A)$.
        \item There exists a continuous and bounded function $r: \mathrm{QT}(A) \longrightarrow (0,\infty)$ such that $\hat{x}(\tau)=d_{\tau}(x) \leq r(\tau) < d_{\tau}(y)=:\hat{y}(\tau)$ for all $\tau \in \mathrm{QT}(A)$.
    \end{enumerate} 
\end{enumerate}

\end{prop}
\begin{proof}
    \textbf{(1)}. Suppose both $x$ and $y$ are compact and $x\neq y$. By Corollary \ref{substraction in Cu(A)}, \mbox{$x\leq y \implies \exists z \in \mathrm{Cu}(A)$} such that $x+z=y$. Observe that if $x=[p]$ and $y=[q]$ for some projections $p$ and $q$ in $(A \otimes\mathcal{K})_{+}$, then $q-p$ is a representative of $z$, that is, $z=[p-q]$. By dominated convergence theorem $d_{\tau}(z)=\tau(q-p)$, so
    \begin{align*}
        d_{\tau}(y) &= d_{\tau}(x) + d_{\tau}(z) = d_{\tau}(x) + \tau(q-p)
    \end{align*}  We claim that any $\tau \in \mathrm{QT}(A)$ is faithful. Indeed, we know $\mathrm{QT}(A)$ is nonempty, for any nonzero quasitrace $\tau \in \mathrm{QT}(A)$, define $N_{\tau}=\{x \in A \mid \tau(x^*x)=0\}$, then $N_{\tau}$ is an ideal of $A$. 
    Since $1_A$ is not in $N_{\tau}$ and $A$ is simple, $N_{\tau}$ must be the zero ideal, and so $\tau$ is faithful. 
    Since $\tau$ is faithful and $x \neq y$, we have $\tau(q-p)>0$, and so $d_{\tau}(x)<d_{\tau}(y)$ for all $\tau \in \mathrm{QT}(A)$.
    
    Now suppose $d_{\tau}(x)<d_{\tau}(y)$ for all $\tau \in \mathrm{QT}(A)$.
    Since $A$ is $\mathcal{Z}$-stable, it has strict comparison, this implies $x \leq y$.

    \vspace{2mm}
    \noindent
    \textbf{(2)}. This follows from Proposition \ref{qt: compact noncompact}.

    \vspace{2mm}
    \noindent
    \textbf{(3)}. Suppose $x$ is noncompact and $y$ is compact. Since $d_{\tau}$ is order-preserving, we have \[x \ll y \implies d_{\tau}(x)\leq d_{\tau}(y) \quad \forall \tau \in \mathrm{QT}(A).\] 
    By Proposition \ref{qt: noncompact compact}, \[x \ll y \impliedby d_{\tau}(x)\leq d_{\tau}(y) \quad \forall \tau \in \mathrm{QT}(A).\]

    \vspace{2mm}
    \noindent
    \textbf{(4)}. It is clear that (b) $\implies$ (c).
    
    Let us show (a) $\implies$ (b). 
    Suppose $x=[a]$ and $y=[b]$ for some $a,b \in (A \otimes\mathcal{K})_{+}$. Let $X=\sigma(b)-\{0\}$, and let us consider the commutative subalgebra $C^*(b)$ generated by $b$, which is isomorphic to $C_0(X)$. 
    Pick an arbitrary $\tau \in \mathrm{QT}(A)$, we know it becomes a tracial state on this commutative subalgebra $C_0(X)$. 
    Let $\mu_\tau$ be the probability measure on $X$ induced by $\tau$ in the Riesz-Markov-Kakutani theorem, that is, for all $f \in C_0(X)$, we have $\tau(f)=\int_{X}fd\mu_{\tau}$. 
    Think of each $f \in C_0(X)$ as a function in $C(\sigma(b))$ such that $f(0)=0$ and apply \mbox{Proposition \ref{rank, measure of suppport}}, we get $d_{\tau}(f(b))=\mu_{\tau}\left(\operatorname{supp}_{\mathrm{o}}(f(b))\right)$ for all $f(b) \in C^*(b)$.
    We have seen in part \textbf{(1)} that $\tau$ is faithful, so the induced measure $\mu_{\tau}$ is also faithful. 
    This is because for any nonempty open set $O \subset X$, $\mu(O)=\int_{X}\chi_{O}d\mu_{\tau}=\tau(\chi_{O})>0$.
    By Proposition \ref{epsilon cc} $x \ll y$ implies there exists $\delta >0$ such that $a \precsim (b-\delta)_+$, and so $d_{\tau}(a) \leq d_{\tau}((b-\delta)_+)$. 
    Since $y=[b]$ is noncompact, we know from  \mbox{Proposition \ref{0 accum pt for noncompact}} that $0\in \sigma(b)$ and $0$ is an accumulation point of $\sigma(b)$. 
    Observe that we can assume $\delta \in X=\sigma(b)-\{0\}$, otherwise we can keep replacing $\delta$ by $\frac{\delta}{2}$ until it is in $X$. 
    Moreover, we can pick two other points $\delta_1$ and $\delta_2$ in $X$ such that \mbox{$0<\delta_2<\delta_1<\delta$}. Then $(\delta_2, \delta_1)$, $(\delta_1, \delta)$, and $(\delta_1, \delta]$ are nonempty  intervals in $X$.
    Now we claim that \[\mu_{\tau}((\delta, \infty) \cap X) < \mu_{\tau}((\delta_1, \infty) \cap X).\] Indeed, as $\mu_{\tau}$ is faithful, we have \[\mu_{\tau}((\delta_1, \infty) \cap X) - \mu_{\tau}((\delta, \infty) \cap X) = \mu_{\tau}(\delta_1, \delta] \geq \mu_{\tau}(\delta_1, \delta) > 0.\]
    Similarly, \[\mu_{\tau}([\delta_1, \infty) \cap X) < \mu_{\tau}((\delta_2, \infty) \cap X).\] 
    Observe that \begin{align*}
        &\tau \mapsto \mu_{\tau}([\delta_1, \infty) \cap X) \quad \text{is upper semicontinuous, while }\\
        &\tau \mapsto \mu_{\tau}((\delta_2, \infty) \cap X)=d_{\tau}((b-\delta_2)_+) \quad \text{is lower semicontinuous.}
    \end{align*}  
    We know when $A$ is unital, $\mathrm{QT}(A)$ is compact in the topology of pointwise convergence.
    By Theorem \ref{continuous btw usc and lsc}, there exists a continuous function  \mbox{$r: QT(A) \longrightarrow (0,\infty]$} such that \[\mu_{\tau}([\delta_1, \infty) \cap X) \leq r(\tau) \leq \mu_{\tau}((\delta_2, \infty) \cap X) \quad \forall \tau \in \mathrm{QT}(A).\] 
    Then we have
    \begin{align*}
        d_{\tau}(a) &\leq d_{\tau}((b-\delta)_+) = \mu_{\tau}((\delta, \infty) \cap X) < \mu_{\tau}((\delta_1, \infty) \cap X) \\ &\leq \mu_{\tau}([\delta_1, \infty) \cap X) \leq r(\tau) \leq \mu_{\tau}((\delta_2, \infty) \cap X) < \mu_{\tau}(X) = d_{\tau}(b) \quad \forall \tau \in \mathrm{QT}(A).
    \end{align*}
    In particular, this shows
    \begin{align*}
        d_{\tau}(a)<r(\tau)<d_{\tau}(b) \quad \forall \tau \in \mathrm{QT}(A).
    \end{align*}

    Now let us show (c) $\implies$ (a). 
    Suppose there exists a continuous function $r: QT(A) \longrightarrow (0,\infty)$ such that \[\hat{x}(\tau)=d_{\tau}(a)\leq r(\tau)<d_{\tau}(b)=\hat{y}(\tau) \quad \forall \tau \in \mathrm{QT}(A)\]
    Let $\{g_n\}_{n \in \N}$ be an increasing sequence of functions from $\mathrm{QT}(A)$ to $(0, \infty]$ such that $\hat{y}(\tau)=\sup_{n \in \N}g_n(\tau)$ for all $\tau \in \mathrm{QT}(A)$. Then by Lemma \ref{cc lemma}, there exists $n_0 \in \N$ such that $g_{n_0}(\tau) > r(\tau) \geq \hat{x}(\tau)$.  and by definition of compact containment, this means $\hat{x} \ll \hat{y}$, and so $x \ll y$.
\end{proof}
\end{prop}

\section{Proof of Theorem \ref{main}}\label{section: proof of main}

In this section we prove our main result, restated here:

\begin{customthm}{1.1}
Let $A$ be unital simple separable $\mathcal{Z}$-stable C$^*$-algebra with nonempty quasitrace simplex $\mathrm{QT}(A)$.  Let $h$ be a continuous affine map from $\mathrm{QT}(A)$ into the set $\mathcal{P}([0,1])^+$ of faithful regular Borel probability measures on $[0,1]$.  It follows that there is a self-adjoint element $a \in A$ with spectrum equal to $[0,1]$ such that $f_a = h$.  
\end{customthm}

Before beginning in earnest let us outline the broad strategy.  We first use the map $h$ to construct a map $\gamma_h:\mathrm{Cu}(C[0,1]) \to \mathrm{Cu}(A)$ and show it is a $\mathbf{Cu}$-morphism. We then use a result of L. Robert to lift $\gamma_h$ to a $*$-homomorphism $\Gamma_h:C[0,1] \to A$ (Theorem \ref{Cu-morphism lifts to $*$-homomorphism}).  Finally, we demonstrate that the image of the identity function under $\Gamma_h$ is a positive operator $a \in A$ such that $f_a = h$.

Recall that 
\[
\mathrm{Cu}(C[0,1]) \cong \mathrm{Lsc}([0,1], \overline{\N})
\]
and
\[
\mathrm{Cu}(A) \cong \mathrm{V}(A) \sqcup \mathrm{LAff(QT}(A){)_{++}}
\]
(see Theorems \ref{Cu(C[0,1]) iso} and \ref{Cu(A) iso}).
Let us define \[\gamma_h: \mathrm{Lsc}([0,1], \overline{\N}) \longrightarrow \mathrm{V}(A) \sqcup \mathrm{LAff(QT}(A){)_{++}}\] as follows:
  \begin{enumerate}
      \item If $f \in \mathrm{Lsc}([0,1], \overline{\N})$ is compact, then we know $f$ is constant with value $n \in \N$ on $[0,1]$ and we define
      \begin{align*}
          \gamma_{h}(f)=n[1_{A}];
      \end{align*}
      \item If $f \in \mathrm{Lsc}([0,1], \overline{\N})$ is noncompact, then we define
      \begin{align*}
          \gamma_{h}(f)(\tau)=\int_{[0,1]}fdh(\tau).
      \end{align*}
  \end{enumerate}
To show $\gamma_h$ is a well-defined $\mathbf{Cu}$-morphism, we need a lemma.

\begin{lem} \label{r(tau) continuous}
    Let $f: [0,1] \longrightarrow [0, \infty]$ be a measurable function, and let $r: \mathrm{QT}(A) \longrightarrow \R$ be defined by $r(\tau)=\int_{[0,1]}fdh(\tau)$. If $f$ is lower semicontinuous, then $r$ is lower semicontinuous; if $f$ is continuous, then $r$ is continuous.
\end{lem}
\begin{proof}
    Let $\{\tau_n\}_{n \in \N}$ be a sequence of points in $\mathrm{QT}(A)$ that converges pointwise to $\tau$ in $\mathrm{QT}(A)$. 
    By continuity of $h$ and Theorem \ref{weak conv equiv to conv in L-P metric}, $h(\tau_n)$ converges to $h(\tau)$ weakly. 
    By the Portmanteau theorem, for every lower semicontinuous function \mbox{$f: [0,1] \longrightarrow [0, \infty]$}, in particular for $f \in \mathrm{Lsc}([0,1], \overline{\N})$, we have \[\liminf \int_{[0,1]}fdh(\tau_n) \geq \int_{[0,1]}fdh(\tau),\] so $r$ is lower semicontinuous. 
    Similarly, if $f$ is continuous, then it is also bounded, and the Portmanteau theorem says $\int_{[0,1]}fdh(\tau_n)$ converges to $\int_{[0,1]}fdh(\tau)$, so $r$ is continuous.
\end{proof}
This shows that when $f \in \mathrm{Lsc}([0,1], \overline{\N})$ is noncompact, $\gamma_{h}(f)$ is indeed lower semicontinuous. It is also easy to see $\gamma_{h}(f)$ is affine when $h$ is affine. Finally, $h(\tau)$ is a faithful measure and so if $f \neq 0$ then there are some $\epsilon>0$ and \mbox{$U \subseteq [0,1]$} open such that $f > \epsilon$ on $U$.  It follows that $\gamma_h(f) > 0$ for every $\tau \in \mathrm{QT}$(A), whence \mbox{$\gamma_{h}(f) \in  \mathrm{LAff(QT}(A){)_{++}}$} and $\gamma_{h}$ is well-defined.

% \begin{lem} \label{strict < integral}
%     Let $f_x$ and $f_y$ be noncompact elements in $\mathrm{Lsc}([0,1], \overline{\N})$ such that $f_x \ll f_y$, then there exists a continuous function $r: \mathrm{QT}(A) \longrightarrow \R$ such that for all $\tau \in \mathrm{QT}(A)$, $\gamma_h(f_x)(\tau) \leq r(\tau) < \gamma_h(f_y)(\tau)$.
% \end{lem}

% \begin{proof}
%     From Lemma \ref{strict < on open}, there exists a continuous $g: [0,1] \longrightarrow [0,\infty)$ such that $f_x \leq g \leq f_y$ on $[0,1]$, and $f_x\leq g<f_y$ on a nonempty open interval $U \subset [0,1]$. Define $r(\tau):=\int_{[0,1]}gdh(\tau)$, we know from Lemma \ref{r(tau) continuous} that $r$ is continuous. Then
%     \begin{align*}
%         r(\tau)&=\int_{[0,1]}gdh(\tau)=\int_{U}gdh(\tau)+\int_{[0,1]-U}gdh(\tau) \\ &<\int_{U}f_ydh(\tau)+\int_{[0,1]-U}f_ydh(\tau)=\int_{[0,1]}f_y dh(\tau)=\gamma_h(f_y)(\tau)
%     \end{align*}
%     Since the Lebesgue integral is order-preserving, $\gamma_h(f_x)(\tau)\leq r(\tau)$ for all $\tau \in \mathrm{QT}(A)$.
% \end{proof}

\begin{prop}
     $\gamma_h$ is a $\mathbf{Cu}$-morphism.
\end{prop}
\begin{proof}
    We need to check that $\gamma_h$ preserves addition, neutral element, order $\leq$, suprema of increasing sequences, and also the compact containment relation $\ll$.

    \vspace{2mm}
    \noindent
    \textbf{(i) Additivity}: If $f$ and $g$ are both compact, say $f=n$ and $g=m$ where $m,n \in \N$, then $\gamma_h(f+g)=(n+m)[1_{A}]=n[1_{A}]+m[1_{A}]=\gamma_h(f)+\gamma_h(g)$.
    If $f$ and $g$ are both noncompact, then for any $\tau \in \mathrm{QT}(A)$, \begin{align*}
        \gamma_h(f+g)(\tau)&=\int_{[0,1]}f+gdh(\tau)\\&=\int_{[0,1]}fdh(\tau)+\int_{[0,1]}gdh(\tau)\\&=\gamma_h(f)(\tau)+\gamma_h(g)(\tau)
    \end{align*}
    If $f$ is compact and $g$ is noncompact, say $f=n\in\N$, then $f+g$ is noncompact, we have \begin{align*}
        \gamma_h(f+g)(\tau)&=\int_{[0,1]}n+gdh(\tau)\\&=\int_{[0,1]}ndh(\tau)+\int_{[0,1]}gdh(\tau)\\&=\widehat{\gamma_h(f)}(\tau)+\gamma_h(g)(\tau)
    \end{align*}
    \vspace{2mm}
    \noindent
    \textbf{(ii) Preservation of zero}: Clearly the zero function on $[0,1]$ is compact, therefore $\gamma_h(0)=0[1_{A}]=0_{\mathrm{Cu}(A)}$.
    
    \vspace{2mm}
    \noindent
    \textbf{(iii) Preservation of order}: If $f$ and $g$ are both compact, say $f=n$ and $g=m$ where $m,n \in \N$, then $f\leq g$ is equivalent to $n\leq m$, and we have $\gamma_h(f)=n[1{A}]\leq m[1_{A}]=\gamma_h(g)$. If $f$ and $g$ are both noncompact, and $f\leq g$, then for all $\tau \in T(A)$ \begin{align*}
        \gamma_h(f)(\tau)=\int_{[0,1]}fdh(\tau)\leq \int_{[0,1]}gdh(\tau)=\gamma_h(g)(\tau)
    \end{align*} If $f$ is compact and $g$ is noncompact, say $f=n\in\N$. 

    \vspace{2mm}
    \noindent
    \textbf{Case 1}: If $f\leq g$ then $g\geq n$ on $[0,1]$. We know $U=\{t \mid g(t) > n\}$ is an open set by lower semicontinuity of $g$. $U$ is also nonempty, otherwise $g$ would be identically $n$ and hence compact.  Then for all $\tau \in \mathrm{QT}(A)$ we have \begin{align*}
        \widehat{\gamma_h(f)}(\tau)&=\int_{[0,1]}ndh(\tau)
        \\&=\int_{U}ndh(\tau)+\int_{[0,1]-U}ndh(\tau) 
        \\&<\int_{U}gdh(\tau)+\int_{[0,1]-U}gdh(\tau)
        \\&=\int_{[0,1]}gdh(\tau) 
        \\&=\gamma_h(g)(\tau)
    \end{align*} 

    \vspace{2mm}
    \noindent
    \textbf{Case 2}: If $g\leq f=n\in \N$, for all $\tau \in \mathrm{QT}(A)$ \begin{align*}
        \gamma_h(g)(\tau)=\int_{[0,1]}gdh(\tau)\leq\int_{[0,1]}ndh(\tau)=\widehat{\gamma_h(f)}(\tau)
    \end{align*}

    \vspace{2mm}
    \noindent
    \textbf{(iv) Preservation of suprema}:
    Let $\{f_n\}_{n\in\N}$ be an increasing sequence of functions in $\mathrm{Lsc}([0,1], \overline{\N})$ and let $f=\sup_{n\in\N}f_n$. 

    \vspace{2mm}
    \noindent
    \textbf{Case 1}: If $f$ is compact, then by definition of compact containment, there exists $n_0\in\N$ such that $f_n=f$ for all $n\geq n_0$. Therefore, we can just assume $\{f_n\}_{n\in\N}$ to be the constant sequence of this compact element $f=N\in\N$. 
    Then \mbox{$\gamma_h(f)=N[1_{A}]=\gamma_h(f_n)$} for all $n\in\N$, and we have $\gamma_h(f)=\sup_{n\in\N}\gamma_h(f_n)$ trivially.

    \vspace{2mm}
    \noindent
    \textbf{Case 2}: If $f$ is noncompact, and only finitely many $f_n$'s are compact, then without loss of generality, we can assume all $f_n$'s are noncompact, as otherwise we can replace $f_n$ by its tail subsequence. Applying monotone convergence theorem, we have
    \begin{align*}
        \gamma_h(f)(\tau)&=\int_{[0,1]}fdh(\tau)
        \\&=\int_{[0,1]}\sup_{n \in \N}f_n dh(\tau)
        \\&=\sup_{n \in \N}\int_{[0,1]}f_ndh(\tau)
        \\&=\sup_{n \in \N}\gamma_h(f_n)(\tau)
    \end{align*} for all $\tau \in \mathrm{QT}(A)$.

    \vspace{2mm}
    \noindent
    \textbf{Case 3}: If $f$ is noncompact, and there are infinitely many distinct compact elements in the sequence $\{f_n\}_{n\in\N}$, let us denote the subsequence of all compact elements by $\{f_{n_k}\}$, where $f_{n_k}=r_{n_k}[1_{A}]$ and $r_{n_k}\in\N$. Observe that $r_{n_k}$ must go to infinity as $f_{n_k}$'s are distinct. Then $f(t)=\sup_{n\in\N}f_n(t)\geq\sup_{n_k\in\N}f_{n_k}(t)=\infty$ for all $t \in [0,1]$. 
    We have shown that $\gamma_h$ is order-preserving, so 
    \[
    d_{\tau}(\gamma_h(f))\geq d_{\tau}(r_{n_k}[1_{A}])=r_{n_k}. 
    \]
    Since $r_{n_k}$ goes to infinity, we have $d_{\tau}(\gamma_h(f))=\infty$ for all $\tau \in \mathrm{QT}(A)$. We know $\mathcal{Z}$-stability implies strict comparison of positive elements, so $\gamma_h(f)=\infty$ (here $\infty$ denotes the maximal element in $\mathrm{Cu}(A)$). On the other hand, 
    \[
    \sup_{n\in\N}\gamma_h(f_n)\geq\sup_{n_k\in\N}\gamma_h(f_{n_k})=\sup_{n_k\in\N}r_{n_k}[1_{A}]=\infty, 
    \]
    so we have $\gamma_h(f)=\sup_{n\in\N}\gamma_h(f_n)$ in this case too.

    \vspace{2mm}
    \noindent
    \textbf{(v) Preservation of compact containment}:
    Finally, we show that $\gamma_h$ preserves compact containment. Suppose $f \ll g$ in $\mathrm{Lsc}([0,1], \overline{\N})$, then we know in particular $f \leq g$.

    \vspace{2mm}
    \noindent
    \textbf{Case 1}: If either $f$ or $g$ is compact, then either $\gamma_h(f)$ or $\gamma_h(g)$ is compact as $\gamma_h$ sends compact elements to compact elements by construction. 
    By Remark \ref{rem: either compact}, to show $\gamma_h(f) \ll \gamma_h(g)$ it suffices to show $\gamma_h(f) \leq \gamma_h(g)$, and this follows from $f \leq g$ and part \textbf{(iii)} that says $\gamma_h$ is order-preserving.

    \vspace{2mm}
    \noindent
    \textbf{Case 2}: If both $f$ and $g$ are noncompact, then we claim there exists a continuous function $r: [0,1] \longrightarrow [0,\infty)$ such that, $\gamma_h(f)(\tau) \leq r(\tau) < \gamma_h(g)(\tau) \quad \forall \tau \in \mathrm{QT}(A)$. Indeed, we know from Lemma \ref{strict < on open}, there exists a continuous $\phi: [0,1] \longrightarrow [0,\infty)$ such that $f \leq \phi \leq g$ on $[0,1]$, and $f\leq \phi<g$ on a nonempty open interval $U \subset [0,1]$. Define \mbox{$r(\tau):=\int_{[0,1]}\phi dh(\tau)$}, we know from Lemma \ref{r(tau) continuous} that $r$ is continuous. Then
    \begin{align*}
        r(\tau)&=\int_{[0,1]}\phi dh(\tau)
        \\&=\int_{U} \phi dh(\tau)+\int_{[0,1]-U} \phi dh(\tau) 
        \\&<\int_{U}gdh(\tau)+\int_{[0,1]-U}gdh(\tau)
        \\&=\int_{[0,1]}g dh(\tau)
        \\&=\gamma_h(g)(\tau)
    \end{align*}
    Since the Lebesgue integral is order-preserving, $\gamma_h(f)(\tau)\leq r(\tau)$ for all $\tau \in \mathrm{QT}(A)$.
    By Proposition \ref{Cu(A) cc}, this implies that $\gamma_h(f) \ll \gamma_h(g)$.
\end{proof}

Now that we know $\gamma_h: \mathrm{Cu}(C[0,1]) \longrightarrow \mathrm{Cu}(A)$ is a $\mathbf{Cu}$-morphism, Theorem \ref{Cu-morphism lifts to $*$-homomorphism} says it lifts to a $*$-homomorphism $\Gamma_h: C[0,1] \longrightarrow A$, that is, $\mathrm{Cu}(\Gamma_h)=\gamma_h$. We also know a $*$-homomorphism $\Gamma_h: C[0,1] \longrightarrow A$ induces a continuous affine map between the spaces of quasitraces $\Gamma_h^{\#}: \mathrm{QT}(A) \longrightarrow \mathrm{QT}(C[0,1])$ where $\Gamma_h^{\#}(\tau)=\tau \circ \Gamma_h$ for all $\tau \in \mathrm{QT}(A)$.  We will show in the proof of Theorem \ref{main} below that $\Gamma_h^{\#} = h$, so that $\gamma_h$ represents a lifting of $h$ to the level of Cuntz semigroups. 

% \begin{thm}
%     Let $A$ be a simple, separable, unital and $\mathcal{Z}$-stable C$^*$-algebra with a nonempty space of quasitraces $\mathrm{QT}(A)$. Let $\mathcal{P}([0,1])$ be the space of faithful Borel probability measures on $[0,1]$ (note that these measures are automatically regular) equipped with Lévy-Prokhorov metric $d_{LP}$, and $h: \mathrm{QT}(A) \longrightarrow \mathcal{P}([0,1])$ be a continuous and affine map. Then there exists a $*$-homomorphism $\Gamma_h: C[0,1] \longrightarrow A$ whose induced measure $\mu(\tau)$ agrees with $h(\tau)$ for all $\tau \in \mathrm{QT}(A)$. In other words, the map $\Gamma_h^{\#}$ induced by $\Gamma_h$ is exactly $h$.
% \end{thm}
\begin{proof}[Proof of Theorem 1.1]
    Define $a=\Gamma_h(id)$, where $id$ represents the identity function on $[0,1]$. It is clear that $a$ is self-adjoint and $\norm{a}=1$, so the spectrum $\sigma(a)$ is inside $[0,1]$. 
    Pick an arbitrary $\tau \in \mathrm{QT(A)}$. 
    We will show $f_{a}(\tau)=h(\tau)$ by showing they agree on any proper open subset $O$ of $[0,1]$. 
    Pick $b \in C[0,1]$ with $\operatorname{supp}_{\mathrm{o}}(b)=O$. 
    We know from Proposition \ref{<= for elements in C(X)} that the above choice is independent of the representative $b$. 
    Let us denote the quasitrace $\Gamma_h^{\#}(\tau)=\tau \circ \Gamma_h$ on $C[0,1]$ by $\Tilde{\tau}$.  Since $C[0,1]$ is commutative, $\Tilde{\tau}$ is automatically a trace. Note that $f_{a}(\tau)$ is actually the probability measure on $\sigma(a)$ induced by $\Tilde{\tau}$ (which eventually depends on $\tau$), so by Proposition \ref{rank, measure of suppport}, $f_{a}(\tau)(O)=f_{a}(\tau)(\operatorname{supp}_{\mathrm{o}}(b))=d_{\Tilde{\tau}}([b])$. 
    We know by Proposition II.1.8 in \cite{BH81} that quasitraces are norm-continuous. Combining this with the continuity of $\Gamma_{h}$ and $b$, we get
    \begin{align*}
        d_{\Tilde{\tau}}([b])=\lim_{n \rightarrow \infty} \Tilde{\tau}(b^{1/n})=\lim_{n \rightarrow \infty} \tau(\Gamma_h(b^{1/n})) =\lim_{n \rightarrow \infty} \tau(\Gamma_h(b)^{1/n})=d_{\tau}([\Gamma_h(b)]).
    \end{align*}
    
     We have $\gamma_h=\mathrm{Cu}(\Gamma_h)$, and $\mathrm{Cu}(\Gamma_h)([b])$ is computed as $\mathrm{Cu}([\Gamma_h(b)])$.  Therefore,
    \begin{align*}
        \gamma_h([b])(\tau)=\mathrm{Cu}(\Gamma_h)([b])(\tau)=\mathrm{Cu}([\Gamma_h(b)])(\tau)=d_{\tau}([\Gamma_h(b)]).
    \end{align*}
    Again, by Proposition \ref{<= for elements in C(X)}, $[b]=\chi_{\operatorname{supp}_{\mathrm{o}}(b)}$, so
    \begin{align*}
        \gamma_h([b])(\tau)=\gamma_h(\chi_{\operatorname{supp}_{\mathrm{o}}(b)})(\tau)=\int_{[0,1]}\chi_{\operatorname{supp}_{\mathrm{o}}(b)}dh(\tau) = h(\tau)(\operatorname{supp}_{\mathrm{o}}(b))=h(\tau)(O).
    \end{align*}
    Therefore, $f_{a}(\tau)(O)=d_{\Tilde{\tau}}([b])=d_{\tau}([\Gamma_h(b)])=\gamma_h([b])(\tau)=h(\tau)(O)$, as required.

    Let us show finally that the spectrum $\sigma(a)$ is exactly the closed unit interval $[0,1]$. 
    Suppose it is not, then by compactness of $\sigma(a)$, there is an open interval $U$ inside $[0,1]-\sigma(a)$. 
    Pick two points $\epsilon$ and $\epsilon'$ inside $U$ such that $0<\epsilon<\epsilon'$. It follows that $(\epsilon,1] \cap \sigma(a)=(\epsilon',1] \cap \sigma(a)$.
    Since $f_{a}$ is the induced probability measure on $\sigma(a)$, we have $f_{a}(\tau)((\epsilon,1])=f_{a}(\tau)((\epsilon',1])$. 
    We have shown that $f_{a}=h$, so this means $h(\tau)((\epsilon,1])=h(\tau)((\epsilon',1])$, and hence $h(\tau)((\epsilon,\epsilon'))=0$, which violates the faithfulness of $h(\tau)$.
    Therefore $\sigma(a)=[0,1]$.
\end{proof}

\bibliographystyle{plain}
\bibliography{main}

\address{Department of Mathematics, Purdue University, 150 N University St, West Lafayette IN, \texttt{atoms@purdue.edu}, \texttt{wan60@purdue.edu}}
% \printbibliography
\end{document}